\documentclass[12pt]{article}
\usepackage{amsmath, amsthm}
\usepackage{thmtools} %

\usepackage{enumitem}

\usepackage{amsfonts}
\usepackage{amssymb}
\usepackage{dirtytalk}
\usepackage{fixltx2e}
\usepackage{esint}
\usepackage{mathtools}

\usepackage{bm}
\usepackage[left=1in,right=1in,top=1in,bottom=1in]{geometry}
\usepackage[hyperref, dvipsnames]{xcolor}
\usepackage[colorlinks=true]{hyperref}

\hypersetup{
  colorlinks=true,
  citecolor=blue,
  linkcolor=purple!80
  }

\renewcommand\d{\mathop{}\!\mathrm{d}}
\newcommand\dx{\d x}
\newcommand\dz{\d z}

\allowdisplaybreaks

\newcommand{\grad}{\nabla}

\newcommand{\BJ}{\mathbf{J}_d}

\newcommand{\norm}[2]{\lVert #1\rVert_{#2}}
\newcommand{\weaklyto}{\rightharpoonup}

\DeclareMathOperator*{\divop}{div}

\numberwithin{equation}{section}



\newcommand{\balpha}{{\bm{\alpha}}}
\newcommand{\bphi}{{\bm{\phi}}}
\newcommand{\R}{\mathbb{R}}
\newcommand{\dist}{\mathrm{dist}}

\newcommand{\diam}[1]{\mathrm{diam}(#1)}

\newcommand{\Jfull}{J} %
\newcommand{\Jd}{J_d} %

\usepackage[nameinlink,noabbrev,capitalize]{cleveref}

\DeclareMathOperator*{\argmin}{argmin}

\newtheorem{theorem}{Theorem}[section]
\newtheorem{prop}[theorem]{Proposition}
\newtheorem{ass}[theorem]{Assumption}

\newtheorem{lemma}[theorem]{Lemma}
\newtheorem{cor}[theorem]{Corollary}

\newtheorem{remark}[theorem]{Remark}

\theoremstyle{definition}



\title{A free boundary problem driven by boundary distance in the coincidence set}

\author{Amal Alphonse\thanks{Email: \url{alphonse@wias-berlin.de}} \and Marcelo Bongarti\thanks{Weierstrass Institute for Applied Analysis and Stochastics, Anton-Wilhelm-Amo-Straße 39, 10117 Berlin. Email: \url{bongarti@wias-berlin.de}} \and Enrico Valdinoci\thanks{Department of Mathematics and Statistics, University of Western Australia,
35 Stirling Highway, WA 6009 Crawley, Australia. Email: \url{enrico.valdinoci@uwa.edu.au}}}
\date{}

\begin{document}
\maketitle

\begin{abstract}
We study a  free boundary problem of minimising a functional containing a non-local term rewarding depth into the zero phase: for $u\geqslant 0$ on a bounded, open set $\Omega\subset\mathbb R^d$, we minimise \[ J(u) = \int_\Omega \left(\frac12|\nabla u|^2 - fu\right) \;-\; \int_{\{u=0\}} F\big(\dist(x,\partial \{u=0\})\big)\,\dx. \]
This kind of functional arises, for example, from a two-membranes problem with an adhesive contact energy. We first address a well-definedness issue caused by the non-local, boundary-sensitive nature of the functional and prove existence of minimisers, establishing along the way a weak lower semicontinuity result for the non-local term. We then derive stationarity conditions for minimisers, including a variational (Euler--Lagrange type) inequality, a PDE on the positivity set, and, under a mild non-degeneracy assumption, a free boundary condition obtained via inner variations and a Danskin-type differentiation of the distance function. %
\end{abstract}
\section{Introduction} 

In this paper, we consider the minimisation of the functional\footnote{All boundaries in this paper are meant in the topological sense (not essential, or measure-theoretic, or reduced).
This technical aspect will play a role in the forthcoming
\cref{OBS:dI}.}
\begin{equation}\label{TERMS}\Jfull(u) = \int_\Omega \left(\frac 12 |\nabla u|^2 - fu\right) - \int_{\{u=0\}} F(\dist(x,\partial \{u=0\}))\dx\end{equation}
over non-negative functions $u\colon \Omega \to \mathbb{R}$ with a prescribed Dirichlet boundary data $u_0$. The first term
on the right-hand side of~\eqref{TERMS}
is the classical Dirichlet energy with forcing $f$, while the second term favours the formation of regions where $u$ vanishes. The contribution of a point in the zero phase depends on its distance to the free boundary, with $F\colon \mathbb{R} \to \mathbb{R}_+$ a non-decreasing function determining how this contribution varies with depth inside the zero region. Here, the notation $\dist(x, A) := \inf_{a \in A} |x-a|$ denotes the (unsigned) distance of the point $x$ to the set $A$. 

In this work, we will study existence of minimisers, derive stationarity conditions including a free boundary condition and provide an explicit characterisation of the free boundary condition in one dimension.

Our motivation for this problem stems from a two-membranes problem where two elastic, non-penetrating membranes $u$ and $v$ are subject to an adhesive energy on their contact set,  where the interaction energy depends on the distance from each contact point to the boundary of the coincidence set: 
\begin{equation*}
\int_\Omega \left( \frac 12 |\grad u|^2 + \frac 12 |\grad v|^2 - fu - gv \right)   -\int_{\{u=v\}}F\left({\rm dist}(x,\partial\{u=v\})\right)\dx.
\end{equation*}
Indeed, the non-local term models an energetic preference for regions in which the two membranes coincide. The present functional does not solely reward coincidence through its size but it weights each point according to its distance from the boundary of the coincidence set. As a consequence, points deep inside a coincidence region contribute more strongly than points near its boundary, so that geometrically thick and robust coincidence regions are energetically preferred over thin or highly fragmented ones, even when they have the same measure. This introduces the effect of coupling the energy to the internal structure of the coincidence set rather than only to its size. Although our motivation stems from elastic membranes, other applications could include those that involve, for example, a synchronised phase in a multiphase system, or a domain where two interacting quantities attain the same state. The two-membranes application will be explored in detail in a future work.

The functional in \eqref{TERMS} resembles the celebrated Alt--Caffarelli functional introduced in the seminal work \cite{AltCaffarelli}  which generated further related works such as \cite{AltCaffarelliFriedman, MR850615}. We refer to \cite{MR4807210} for a survey of one-phase free boundary problems covering existence and various regularity aspects. We mention here a few related works: the vectorial analogue of the Alt--Caffarelli functional is the focus in \cite{MR3827812, MR3626615}; biharmonic problems are analysed in \cite{MR4229229, MR4816111}; the recent work \cite{2025arXiv250814736A} considers borderline regularity for a class of related functionals; the paper \cite{McCurdy} focusses on a functional also consisting of a distance function, with the distance being from points in the positive phase to a given $C^{1,\alpha}$ submanifold.

Two difficulties in our problem are worth emphasising:
\begin{enumerate}[itemsep=0pt, label=(\roman*)]
    \item \textbf{well-definedness}: since we take \textit{boundaries} of level sets, one has to be careful in utilising the correct function spaces and definitions that allow for that term to be well defined;
    \item \textbf{non-smoothness}: the non-differentiability of the distance function requires a careful argument in order to derive the free boundary condition.
\end{enumerate}
Furthermore, the \textbf{non-locality} in the functional is challenging: specifically, the non-locality of the distance function implies that standard arguments from the free boundary literature cannot straightforwardly be adapted (e.g. in order to show continuity of the minimisers).

As indicated above, $J$ is not well defined over $H^1(\Omega)$, since, in general, the set $\partial\{u=0\}$ depends on the representative chosen for the Sobolev function and hence is not well defined for $H^1$-functions. For example, in 2D the functions $u \equiv 0$ and $\tilde u$, defined via
\[\tilde u(x) = \begin{cases}
0 &: x \in \Omega \setminus \mathbb{Q}^2,\\
1 &: x \in \Omega  \cap \mathbb{Q}^2,
\end{cases}
\]
can both be identified with the zero function and are such that $\{u=0\}$ and $\{\tilde u = 0\}$ agree a.e., %
yet $\partial \{u=0\} = \partial\Omega$ and $\partial\{\tilde u = 0\} = \overline\Omega$, i.e., the two representatives lead to different values of the functional. Indeed, this example also serves to demonstrate that e.g. choosing a quasi-continuous representative does not alleviate the problem as these are %
unique only up to sets of capacity zero, and two different representatives of the same function could give rise to radically different boundaries.

To allow for a well-defined problem in $H^1(\Omega)$, we define the zero set $Z\colon H^1(\Omega) \to \mathcal{P}(\mathbb{R}^d)$ by
\begin{equation}
 Z(u) = \left\{ x \in \Omega : \liminf_{r \to 0}\fint_{B_r(x)} u = 0\right\}.\label{eq:defnZ}
\end{equation}
Hence, instead of~\eqref{TERMS}, we consider (without changing notation)
\begin{equation}
\Jfull(u) :=  \int_\Omega \left(\frac 12 |\nabla u|^2 - fu  \right)- \int_{\{u=0\}} F(\dist(x,\partial Z(u)))\;\mathrm{d}x\label{eq:min_problem}
\end{equation}
and we focus on the minimisation of this functional over the
admissible set
\begin{equation}
    {K} := \{ u \in H^1(\Omega) : u \geqslant 0 \text{ in $\Omega$, }  u = u_0 \text{ on $\partial\Omega$}\}.
\end{equation}
Here, the boundary condition is meant in the sense that $u-u_0 \in H^1_0(\Omega)$. The precise assumptions that we make on the data are as follows, which we take to be standing throughout the entire work:
\begin{ass}[Standing assumptions]~
\begin{enumerate}[itemsep=0pt, label=(\roman*)]
\item  $\Omega \subset \mathbb{R}^d$ is a bounded, open set for $d \geqslant 1$; 
\item  $f \in L^2(\Omega)$;
\item  $u_0 \in H^1(\Omega)$;
\item  $F\colon \mathbb{R}_+ \to \mathbb{R}_+$ is non-decreasing and Lipschitz continuous.
\end{enumerate}
\end{ass}
Let us state the main results of our paper.
\begin{theorem}[Existence of minimisers]\label{thm:existence}
There exists a minimiser $u\in{K}$ of problem~\eqref{eq:min_problem}.
\end{theorem}
The proof of this result will be given in \cref{sec:existence}, where in particular we work to establish the weak lower semicontinuity of our functional $J$. 

We note that if a minimiser $u$ of \eqref{eq:min_problem} is continuous, then $Z(u) = \{u=0\}$. In this case, we can relate our existence result to the original functional \eqref{TERMS} as follows.
\begin{cor}
If the minimiser $u$ of $J$ is continuous, then $u$ minimises the original functional
\[\int_\Omega \left(\frac12|\nabla u|^2 - fu\right) \;-\; \int_{\{u=0\}} F\big(\dist(x,\partial \{u=0\})\big)\dx\]
over $K \cap C^0(\Omega)$.
\end{cor}

In analogy to the zero level set $Z(u)$ defined in \eqref{eq:defnZ}, we define also the positivity set
\begin{equation}\label{PIUde}
 P(u) = \left\{ x \in \Omega : \liminf_{r \to 0}\fint_{B_r(x)} u > 0\right\}.
 \end{equation}
As is standard in free boundary problems, we can also show that minimisers satisfy a PDE on the positivity set (i.e. non-coincidence set) in the form of the following:
\begin{prop}\label{prop:pde_solved_on_Pu}
If $u$ is a minimiser of $J$,  we have
that\begin{equation}\label{3.2.2}
-\Delta u = f  \quad \text{on $\mathrm{int}(P(u))$}.
\end{equation}
\end{prop}
See \cref{sec:stationarity} for the proof. To obtain free boundary conditions, we need a regularity assumption on the positivity set.
\begin{ass}\label{ass:measure_Pu_zero}
For a minimiser $u$ of $J$, assume that
\[
|\partial P(u)|=0.
\]
    \end{ass}
    This is a reasonable assumption and it ensures that the free boundary is not too pathological. In \cref{sec:fbc}, we shall address the following theorem.
\begin{theorem}[Free boundary condition] \label{thm:fbc}
Let $u$ be a minimiser of $J$ satisfying \cref{ass:measure_Pu_zero}. Assume also that $\partial\{u > \varepsilon\}$ is of class $C^1$ for sufficiently small $\varepsilon>0$. Define $m_u(x) =\dist(x,\partial Z(u))$. Then, for any $\bphi \in C_0^1(\Omega,\mathbb{R}^d)$,
\begin{align} 
    \lim_{\varepsilon \downarrow 0} \int_{\partial \{u>\varepsilon\}} \frac 12|\nabla u|^2  \bphi \cdot \nu & = -\int_{\{u=0\}} \left((F' \circ m_{u}) \nabla m_u \cdot \left[\bphi(x)-\bphi(x_p)\right] + (F \circ m_u)\divop \bphi(x)\right)\dx \nonumber \\ 
    &=  -\lim_{\varepsilon \to 0} \int_{\partial \{u > \varepsilon\}} (F \circ m_u) \bphi \cdot \nu + \int_{\{u=0\}} (F' \circ m_{u}) \nabla m_u \cdot \bphi(x_p)\dx\label{fb_0}  
 \end{align} 
    where we set $x_p := \argmin_{z \in \overline{P(u)} \cup \partial\Omega} |x-z|$  and $\nu = -\frac{\grad u}{|\grad u|}$ denotes the normal vector to $\partial\{u>\varepsilon\}$. 
    \end{theorem}
This formula is rather implicit because it contains the point $x_p$, which is essentially a closest point projection onto the set $\overline{P(u)}\cup\partial\Omega$. This set admits a direct interpretation: it is the closure of every point at which the zero phase $Z(u)$ is obstructed from extending --- either because $u$ becomes positive there, or because the domain itself ends there. Its interior, $P(u)\cup\partial\Omega$ (which we shall later define as $\hat P(u)$), will indeed be essential to our analysis.

We can write the free boundary condition more explicitly in dimension $d=1$ in a setting where the positivity set is a union of disjoint intervals. In this way we characterise the 1D setting with a \textit{pointwise} free boundary condition.

\begin{cor}[Pointwise free boundary condition, $d=1$]\label{cor:1D_FB}
Assume $\Omega=(0,1)$. Let \cref{ass:measure_Pu_zero} hold for a minimiser $u \in C^1(\overline\Omega)$ of $J$, and suppose that
\[
\{u>0\}=\bigcup_{i=1}^k(a_i,b_i)
\]
is a finite union of disjoint intervals with $0< a_1 < b_1 < \dots < a_k < b_k < 1$. 
Then, we have the following pointwise free boundary conditions for $i=1, \dots, k$:
\begin{align*}
\frac12u'(a_i)^2 &=F\left(\dfrac{a_i-b_{i-1}}{2}\right),\\
\frac12u'(b_i)^2 &=F\left(\dfrac{a_{i+1}-b_i}{2}\right),
\end{align*}
where $b_0:=0$ and $a_{k+1} :=1$.
\end{cor}

\section{Structural properties of the functional and of the zero and positivity sets}

In this section, we collect some basic features of the problem under consideration. Recall again the zero and positivity sets defined in \eqref{eq:defnZ} and \eqref{PIUde}.
 \begin{lemma}
     Suppose $u \in H^1(\Omega)$ with $u \geqslant 0$ a.e. Then, $\Omega = P(u) \cup Z(u)$.
 \end{lemma}
 \begin{proof}
Since both~$P(u)$ and~$Z(u)$ are subsets of~$\Omega$,
one inclusion is trivial. For the other, we argue by contradiction: if $x \in \Omega$ but $x \not\in P(u) \cup Z(u)$, then $\liminf_{r \to 0} \fint_{B_r(x)}u < 0$, but this cannot be, because~$u \geqslant 0$. 
 \end{proof}

It is useful to observe that the zero and positivity sets $Z(u)$ and $P(u)$ indeed capture, essentially, the sets in which~$u$ is zero and positive respectively. More precisely, we have that:

\begin{lemma}\label{lem:Pu_and_u_positive_same}
    $P(u) = \{ u > 0\}$ and $Z(u) = \{u = 0\}$, up to sets of measure zero.
\end{lemma}
\begin{proof}
    By the Lebesgue differentiation theorem, $$\liminf_{r \to 0}\fint_{B_r(x)} u = \lim_{r \to 0}\fint_{B_r(x)} u = u(x)$$ for almost every $x \in \Omega$, hence the two sets above agree a.e.
\end{proof}
\begin{cor}\label{lem:Pu_cap_u_zero_measure_zero}
    The sets $\{u > 0\} \cap Z(u)$ and $\{u = 0\} \cap P(u)$ have measure zero.
\end{cor}
\begin{proof}
    By \cref{lem:Pu_and_u_positive_same}, we know that~$P(u)\subseteq\{u > 0\}\cup N$, with $N$ of measure zero. Consequently, $$
    \{u = 0\}\cap P(u)\subseteq \big( \{u = 0\}\cap\{u > 0\}\big)\cup N=N,$$
    as desired. A similar argument proves the statement for the zero level set.
\end{proof}

\subsection{Well-definedness}
We show here that $J$ is well defined. 

\begin{lemma}
Suppose $u \in H^1(\Omega)$ with $u \geqslant 0$ a.e., and let $v$ be equal to $u$ almost everywhere in $\Omega$. Then,
$Z(u)=Z(v)$ and $P(u)=P(v)$.\end{lemma}

\begin{proof} We have that, for all~$r>0$,
\[ \fint_{B_r(x)} u=\fint_{B_r(x)} v,\]
from which the desired result follows plainly.\end{proof}

\subsection{Reformulation of the distance functional} %
In this section, we will reformulate the functional in a more convenient form by rewriting the distance to the boundary of the zero level set in terms of the positivity set. First, we recall a useful observation on the distance function.\footnote{Here it is important that $\partial A$ is not given the subspace topology.
}

\begin{lemma}\label{OBS:dI}
Let $A \subset \mathbb{R}^d $ be nonempty. Then, for all $x \in A^c$,
$$\dist(x,A) = \dist(x,\partial A).$$
\end{lemma}

\begin{proof} We distinguish two cases, either $x\in \partial A$ or~$x$ lies
in the interior of~$A^c$.

In the first case, there exists a sequence of points~$a_k\in A$ such that~$a_k\to x$ as $k\to+\infty$. Therefore,
\[0\leqslant\dist(x,A)\leqslant\lim_{k\to+\infty}|x-a_k|=0,\]
showing that~$\dist(x,A)=\dist(x,\partial A)=0$, as desired.

If instead $x$ lies
in the interior of~$A^c$, we can find~$r\in(0,+\infty)$ such that~$B_r(x)\subset A^c$ and~$\partial B_r(x)\cap \partial A$ is nonempty.
Thus, $r=\dist(x,A) = \dist(x,\partial A)$. 
\end{proof}

The next observation is structurally very useful.
\begin{lemma}\label{lem:Z_reduction}
For every $x\in Z(u)$,
\[
\dist(x,\partial Z(u)) = \dist(x, P(u) \cup \partial\Omega).
\]
\end{lemma}
\begin{proof}
 Define $N:=\mathbb R^d\setminus Z(u)$. Then we have
 \begin{align*}
N &= \mathbb{R}^d\setminus (\Omega\setminus P(u)) = (\mathbb{R}^d\setminus \Omega)\cup P(u)  = P(u)\cup\partial\Omega\cup(\mathbb{R}^d\setminus\overline{\Omega}).
\end{align*}
This set is nonempty since $\partial\Omega\neq\emptyset$. Since the distance to a union is the minimum of the distances to each component,
\[
\dist(x,N) = \min\Big(\dist(x,P(u)),\ \dist(x,\partial\Omega),\ \dist\big(x,\mathbb R^d\setminus\overline\Omega\big)\Big).
\]
The third term is not less than $\dist(x,\partial\Omega)$ (because the distance of a point in the set to the boundary cannot be greater than distance to the exterior) so it  can be dropped, leading to
\begin{equation}\label{CSLMM}
\dist(x,N) = \min\big(\dist(x,P(u)),\ \dist(x,\partial\Omega)\big).
\end{equation}
Furthermore, since 
the boundary of a set is invariant under complementation, we have that
$\partial N=\partial(N^c)=\partial Z(u)$.
Thus,
given $x\in Z(u)=N^c$, \cref{OBS:dI} gives that $\dist(x,\partial Z(u))=\dist(x,N)$, which combined with~\eqref{CSLMM} proves that 
\[\dist(x,\partial Z(u)) = \min\big(\dist(x,\partial\Omega),\ \dist(x,P(u))\big).\]
Using again that the distance to a union of sets equals the minimum of the distances to each set, we get the result. 
\end{proof}
This result is quite helpful because it allows us to write our functional $J$ in \eqref{eq:min_problem} as
\[ \Jfull(u) =  \int_\Omega \left(\frac 12 |\nabla u|^2 - fu\right)  - \int_{\{u=0\}} F(\dist(x, P(u) \cup \partial\Omega))\;\mathrm{d}x.\]
In the next sections we will be mainly working with this reformulation.
For shortness, we define
\begin{equation}
    \hat P(u) := P(u) \cup \partial\Omega.\label{eq:hatP}
\end{equation}

\begin{lemma}
For a.e. $x \in \{u=0\}$,
\begin{equation}\label{CL1}
\dist(x, \partial P(u)) = \dist(x, P(u)).\end{equation}
Also,\footnote{The distance to the empty set is defined to be~$\infty$: in this spirit,
both equations~\eqref{CL1}
and~\eqref{C2} read ``$\infty=\infty$'' when $u \equiv 0$. }
\begin{equation}\label{C2}
\int_{\{u=0\}}F(\dist(x, \partial P(u)))\dx = \int_{\{u=0\}}F(\dist(x, P(u)))\dx.\end{equation}
\end{lemma}

\begin{proof} We let $A:=P(u)$.
By \cref{lem:Pu_cap_u_zero_measure_zero}, we know that
almost every~$x\in \{u=0\}$ belongs to~$A^c$ and thus \eqref{CL1}
is a consequence of \cref{OBS:dI}.

Now we remark that, if $f=g$ a.e. in $\Omega$, then
$$ \{x\in\Omega {\mbox{ s.t. }} f(x)= g(x)\}
\subseteq
\{x\in\Omega {\mbox{ s.t. }} F(f(x))= F(g(x))\}$$
and so, passing to complementary sets,
$$ \{x\in\Omega {\mbox{ s.t. }} F(f(x))\ne F(g(x))\}\subseteq
\{x\in\Omega {\mbox{ s.t. }} f(x)\ne g(x)\}.$$
In particular, choosing~$f(x):=\dist(x, \partial P(u))$ and $g(x): = \dist(x, P(u))$, we find that the set
$$ \{x\in\Omega {\mbox{ s.t. }} F(\dist(x, \partial P(u)))\ne F(\dist(x, P(u)))\}$$
has measure zero, thanks to \eqref{CL1}.

On this account,  the set
$$ \{x\in\Omega {\mbox{ s.t. }} F(\dist(x, \partial P(u)))= F(\dist(x, P(u)))\}$$has full measure in~$\Omega$.
The claim in~\eqref{C2} then follows by integration.
\end{proof}

\begin{remark}    
It is natural to ask what happens if one were to consider in $J$ the distance
\[
\dist\big(x,\ \partial\{u=0\}\cap\Omega\big)
\]
in place of $\dist(x,\partial\{u=0\})$.
In this respect, since
\[
\partial\{u=0\} \cap \Omega \;=\; \partial\{u>0\} \cap \Omega,
\]
one has that
$$ \int_{\{u=0\}}F(\dist(x,\partial\{u=0\}\cap\Omega))\dx= \int_{\{u=0\}} F(\dist(x,\partial\{u>0\}\cap\Omega))\dx.$$
Note that $\partial\{u>0\}\cap\Omega$ can be empty even when $\{u>0\}\neq\emptyset$ (e.g. if $\{u>0\}$ is the entire $\Omega$, or connected components of~$\Omega$, and the boundary conditions allow it), and it is certainly empty when $u\equiv 0$. 

In such a case $\dist(x,\partial\{u=0\}\cap\Omega)=+\infty$ for every $x\in\Omega$, and the corresponding term in the functional is ill-defined.
 This clarifies the role of the ambient (non-intersected) boundary $\partial\{u=0\}$ used throughout the paper: it is precisely the retention of the $\partial\Omega$-contribution  that guarantees the uniform bound $\dist(x,\partial\{u=0\})\leqslant \diam{\Omega}$ for every admissible $u$, including $u\equiv0$. 
 
 One could possibly alleviate these problems by assuming, for example, that $F$ is bounded at $\infty$ or by treating
 separately the case~$u\equiv0$,
but the setting considered in this paper avoids these additional complications and allows a unified treatment.
\end{remark}
 
\section{Existence of minimisers}\label{sec:existence}
In this section, we will prove existence of minimisers by establishing a weak lower semicontinuity result for the functional.

\subsection{Approximating problem}\label{sec:approximating_problem}
To prove weak lower semicontinuity of $J$, given a weakly convergent sequence $u_n \weaklyto u$, one would like to compare $\dist(x,P(u_n))$ and $\dist(x,P(u))$ directly, but weak convergence alone need not give enough control on the positivity sets to do so. We instead work outside a small exceptional set on which $u_n\to u$ uniformly, obtained via Egorov's theorem, and this uniform control is what makes the comparison possible. For this purpose, we work with the following set up.

Given $u \in H^1(\Omega)$, let us take a sequence $u_n \weaklyto u$ in $H^1(\Omega)$, strongly in $L^2(\Omega)$ and $u_n, u \geqslant 0$. By Egorov's theorem, for every $\rho$ there exists $E_\rho \subset \Omega$ with \begin{equation}\label{JISKND1}|E_\rho| \leqslant \rho\end{equation} and 
\begin{equation}\label{JISKND}
    {\mbox{$u_n \to u$ uniformly in $\Omega\setminus E_\rho$.}}\end{equation} Furthermore, we can assume without loss of generality that if
$\sigma < \rho$ then $ E_\sigma \subset E_\rho$,
i.e., that the exceptional sets are nested. Now, define
\begin{equation}\label{DEPR}
P_\rho(u) = \left\{ x \in \Omega : \liminf_{r \to 0} \frac{1}{|B_r(x)|}\int_{B_r(x)\setminus E_\rho} u > 0\right\}\end{equation}
and note that we are dividing by the measure of $B_r(x)$, not the area of integration. 

\begin{lemma}\label{LE24}
There exists $N_0(\rho,x) \in \mathbb{N}$ such that if $x \in P_\rho(u)$ then $x \in P_\rho(u_n)$ for $n \geqslant N_0(\rho,x)$.
\end{lemma}
\begin{proof}
Let $x \in P_\rho(u)$ and set $$ L(x):= \liminf_{r  \to 0} \frac{1}{|B_r(x)|}\int_{B_r(x)\setminus E_\rho} u.$$ By~\eqref{DEPR}, we know that $L(x) > 0.$ Hence, we can pick $\varepsilon(x)\in(0, L(x))$. We have that
\begin{align*}
    \int_{B_r(x)\setminus E_\rho} u_n =     \int_{B_r(x)\setminus E_\rho}( u + u_n - u )\geqslant \int_{B_r(x)\setminus E_\rho} u - |B_r(x)\setminus E_\rho|\norm{u_n - u}{L^\infty(B_r(x)\setminus E_\rho)}
\end{align*}
and thus
\begin{equation}\label{EGR1}
\frac{1}{|B_r(x)|}    \int_{B_r(x)\setminus E_\rho} u_n \geqslant \frac{1}{|B_r(x)|}\int_{B_r(x)\setminus E_\rho} u - \frac{|B_r(x)\setminus E_\rho|}{|B_r(x)|}\norm{u_n - u}{L^\infty(B_r(x)\setminus E_\rho)}.
\end{equation}
Now, in light of \eqref{JISKND}, let us take $n \geqslant N_0(\rho,x)$ such that 
$$\norm{u_n - u}{L^\infty(B_r(x)\setminus E_\rho)}\leqslant 
\norm{u_n - u}{L^\infty(\Omega\setminus E_\rho)}\leqslant\frac{\varepsilon(x)}2.$$
Using this and the fact that~$|B_r(x)\setminus E_\rho|\leqslant|B_r(x)|$, we deduce from~\eqref{EGR1} that
\[\frac{1}{|B_r(x)|}    \int_{B_r(x)\setminus E_\rho} u_n \geqslant \frac{1}{|B_r(x)|}\int_{B_r(x)\setminus E_\rho} u - \frac{\varepsilon(x)}{2}.
\]
Taking the liminf here, we obtain 
\[\liminf_{r \to 0}\frac{1}{|B_r(x)|} \int_{B_r(x)\setminus E_\rho} u_n \geqslant L(x) - \frac{\varepsilon(x)}{2}>0,\]
as desired.
\end{proof}

Recall that  $\hat P(u) = P(u) \cup \partial\Omega$. We define similarly 
\[\hat P_\rho(u) := P_\rho(u) \cup \partial\Omega.\]

\begin{prop}\label{PRO25}
 For all $\varepsilon>0$ sufficiently small we have
\[\displaystyle\limsup_{n\to\infty}\int_{\{u<\varepsilon\}\setminus E_\rho} F(\dist(x,\hat P_\rho(u_n)))\dx \leqslant \int_{\{u<\varepsilon\}\setminus E_\rho} F(\dist(x,\hat P_\rho(u)))\dx.\]
 \end{prop}
\begin{proof}
Define
\[
I_n(x) := F \circ \dist(x, \hat P_\rho(u_n)) \, \chi_{\{u < \varepsilon\}\setminus E_\rho}(x), \quad
I(x) := F \circ \dist(x, \hat P_\rho(u)) \, \chi_{\{u < \varepsilon\}\setminus E_\rho}(x).
\]
If $x \in P_\rho(u)$, then, by \cref{LE24}, for $n \geqslant N_0(\rho,x)$, we have that~$x \in P_\rho(u_n)$, and clearly this property is preserved when taking an union with $\partial\Omega$, so that, on the one hand, if $x \in \hat P_\rho(u)$,
\[\dist(x, \hat P_\rho(u_n)) = 0 = \dist(x, \hat P_\rho(u))\quad{\mbox{ for all }} n \geqslant N_0(\rho,x).\]
On the other hand suppose that $x \notin \hat P_\rho(u)$. Recall
\[
\dist(x, \hat P_\rho(u)) = \inf_{z \in \hat P_\rho(u)} |x - z|.
\]
By definition of the infimum, for any $\delta > 0$, we can find $y_\delta(x) \in \hat P_\rho(u)$ such that
\[
|x - y_\delta(x)| \leqslant \dist(x, \hat P_\rho(u)) + \delta.
\]
By \cref{LE24}, there exists $N_1(\rho,y_\delta(x))$ such that
\[
y_\delta(x) \in \hat P_\rho(u_n) \quad \text{for all } n \geqslant N_1(\rho,y_\delta(x)).
\]
Hence, for $n \geqslant N_1(\rho,y_\delta(x))$,
\[
\dist(x, \hat P_\rho(u_n)) = \inf_{z \in \hat P_\rho(u_n)} |x - z| \leqslant |x - y_\delta(x)| \leqslant \dist(x, \hat P_\rho(u)) + \delta.
\]
Taking the limsup as $n \to \infty$, and then using that $\delta > 0$ was arbitrary, we conclude
\[
\limsup_{n\to\infty} \dist(x, \hat P_\rho(u_n)) \leqslant \dist(x, \hat P_\rho(u)).
\]
Therefore, applying $F$ to the above inequality, 
\begin{equation}\label{CNl}
    \limsup_{n\to\infty} I_n(x) \leqslant I(x).
\end{equation}
Now we apply Fatou's lemma:
\[
\limsup_{n\to\infty} \int_\Omega I_n \leqslant \int_\Omega \limsup_{n\to\infty} I_n \leqslant \int_\Omega I,
\]
with the final inequality due to \eqref{CNl}.
\end{proof}

\begin{prop}\label{prop:wlsc_like_rho}
    For each $\rho  >0$,  we have
 $$\displaystyle\liminf_{n\to\infty}-\int_{\{u_n=0\}\setminus E_\rho} F(\dist(x,\hat P_\rho(u_n)))\dx \geqslant -\int_{\{u=0\}\setminus E_\rho} F(\dist(x,\hat P_\rho(u)))\dx.$$
\end{prop}
\begin{proof}
Note that for every $\varepsilon >0 $, there exists $\bar N$ such that $n \geqslant \bar N$ implies $\{u_n=0\} \setminus E_\rho \subset \{ u < \varepsilon \} \setminus E_\rho$ uniformly in $x$, thanks to the uniform convergence outside $E_\rho$. Then we have, for such $n$,
    \begin{align*}
        \int_{\{u_n=0\}\setminus E_\rho} F(\dist(x, \hat P_\rho(u_n)))\dx \leqslant         \int_{\{u < \varepsilon\}\setminus E_\rho} F(\dist(x, \hat P_\rho(u_n)))  
    \dx.\end{align*}
    Now take the limsup as $n \to \infty$ and use \cref{PRO25} on the right-hand side to get
    \begin{align*}
        \limsup_{n \to \infty}\int_{\{u_n=0\}\setminus E_\rho} F(\dist(x, \hat P_\rho(u_n)))\dx \leqslant      \int_{\{u < \varepsilon\}\setminus E_\rho} F(\dist(x, \hat P_\rho(u)) )\dx
    \end{align*}    
and then the limit as $\varepsilon \to  0$ and use the dominated convergence theorem on the right-hand side above to find
        \begin{align*}
        \limsup_{n \to \infty}\int_{\{u_n=0\}\setminus E_\rho} F(\dist(x, \hat P_\rho(u_n)))\dx  \leqslant \int_{\{u =0\}\setminus E_\rho} F(\dist(x, \hat P_\rho(u)))\dx, 
    \end{align*}
which is precisely the statement.

\end{proof}

\subsection{Passing to the limit in \texorpdfstring{$\rho$}{rho} and weak lower semicontinuity}

In this section we discuss the asymptotics as~$\rho\downarrow0$
and we use it to deduce the lower semicontinuity of the functional. ~%

\begin{lemma}\label{lem:Prho_monotone}
    We have $P_\rho(u) \subset P(u)$ and hence $\dist(x, P(u)) \leqslant \dist(x, P_\rho(u)).$
    
    In addition, $\hat P_\rho(u) \subset \hat P(u)$ and hence $\dist(x, \hat P(u)) \leqslant \dist(x, \hat P_\rho(u)).$
\end{lemma}
\begin{proof}
    Since $B_r(x)\setminus E_\rho \subset B_r(x)$ and $u \geqslant 0$, we have
    \begin{align*}
        \fint_{B_r(x)} u = \frac{1}{|B_r(x)|} \int_{B_r(x)} u \geqslant \frac{1}{|B_r(x)|}\int_{B_r(x)\setminus E_\rho} u,
    \end{align*}
    which, after taking the liminf, is positive if $x \in P_\rho(u)$. The result follows.
\end{proof}

\begin{lemma}\label{nested}
    If $\sigma < \rho$ then $P_\rho(u) \subset P_\sigma(u)$ and $\hat P_\rho(u)\subset\hat P_\sigma(u)$.
\end{lemma}
\begin{proof}
    This is clear due to
\begin{align*}
    \frac{1}{|B_r(x)|}\int_{B_r(x) \setminus E_\sigma} u \geqslant     \frac{1}{|B_r(x)|}\int_{B_r(x) \setminus E_\rho} u,
\end{align*}
which holds since we took a nested sequence of exceptional sets.
\end{proof}

\begin{lemma}\label{union_P}
    For any sequence $\rho_k \downarrow 0$ we have that $$P(u) = \bigcup_{k} P_{\rho_k}(u)$$
and
$$\hat P(u)=\bigcup_k\hat P_{\rho_k}(u)$$    
up to sets of measure zero.
\end{lemma}
\begin{proof}
For the first statement, it suffices to show that \begin{equation}\label{SID2}
        P(u) \subseteq \bigcup_{k} P_{\rho_k}(u)
    \end{equation} since, from \cref{lem:Prho_monotone}, we know that~$P_\rho(u) \subseteq P(u)$ for all $\rho > 0$. 
    
    In fact, more is true: denoting $\Omega_\rho = \Omega \setminus E_\rho$ we claim that \begin{equation}\label{SID}P(u) \cap \Omega_\rho \subset P_\rho(u).\end{equation} Indeed, the Lebesgue differentiation theorem shows that
    \[
    \lim_{r \to \infty} \fint_{B_r(x)} u = u(x) \ \text{ and } \ \lim_{r \to \infty} \fint_{B_r(x)} u \chi_{E_\rho} = u(x)\chi_{E_\rho}(x)
    \]
    for almost every $x \in \Omega.$ Hence, for almost every~$x \in P(u) \cap \Omega_\rho$ we have
    \[
    \liminf_{r\to 0} \fint_{B_r(x)} u\chi_{E_\rho^c} = \liminf_{r=0}\left(\fint_{B_r(x)} u - \fint_{B_r(x)} u\chi_{E_\rho}\right) = u(x) - u(x)\chi_{E_\rho}(x) = u(x)>0,
    \]where the last inequality is a consequence of \cref{lem:Pu_and_u_positive_same}. This establishes~\eqref{SID}.

    Now, since $\rho_k \downarrow 0$ and the sets $E_{\rho_k}$ are nested and decreasing, it follows that for almost all $x \in \Omega$, there exists $\overline k(x)$ such that $x \not\in E_{\rho_{\overline k(x)}}$ (this is an easy consequence of the fact that $|\cap_k E_{\rho_k}| = 0$). 
    
    Hence, if $x \in P(u)$ and $x \not\in E_{\rho_{\overline k(x)}}$, then $x \in P(u) \cap \Omega_{\rho_{\overline k(x)}} \subset P_{\rho_{\overline k(x)}}(u).$
    This finishes the proof of \eqref{SID2}.
    
    For the second statement, we see that
    $$P(u)\cup\partial\Omega =\left(\bigcup_k P_{\rho_k}(u)\right)\cup\partial\Omega = \bigcup_k\big(P_{\rho_k}(u)\cup\partial\Omega\big) $$ up to a null set, where we used the first proven statement.
\end{proof}

\begin{prop}\label{prop:distance_convergence}
        We have that
    $$\lim_{\rho\downarrow0}\dist(x,\hat P_\rho(u)) = \dist(x,\hat P(u))$$ a.e., and hence $$\lim_{\rho\downarrow0}F(\dist(x,\hat P_\rho(u))) = F(\dist(x,\hat P(u)))$$ in $L^1(\Omega)$.
\end{prop}
\begin{proof}
It follows from \cref{nested} that if $\rho_k \downarrow 0$, the sets $\hat P_{\rho_k}(u)$ are increasing. 

Now note the following fact: if $A_n \nearrow A$ and $\cup A_n = A$, then $\dist(x,A_n) \searrow \dist(x, A)$. Thus, $\dist(x, \hat P_{\rho_k}(u)) \searrow \dist(x, \cup_k \hat P_{\rho_k}(u)) = \dist(x, \hat P(u))$ the last equality being due to \cref{union_P}. The dominated convergence theorem and the
boundedness of $\Omega$ give the result.
\end{proof}

\begin{theorem}\label{thm:wlsc_Z_simple}
The map 
\[u\mapsto-\int_{\{u=0\}}F(\dist(x,\partial Z(u)))\dx\]
is weakly lower semicontinuous on $K$. %
\end{theorem}
\begin{proof}
As mentioned, by \cref{lem:Z_reduction}, it suffices to show that $u\mapsto-\int_{\{u=0\}}F(\dist(x,\hat P(u)))\dx$ is weakly lower semicontinuous.  

Given $u \in H^1(\Omega)$, take a sequence $u_n \weaklyto u$ in $H^1(\Omega)$ and sets $E_\rho$ satisfying the properties\footnote{Note that $u_n \to u$ strongly in $L^2(\Omega)$ follows without further regularity of $\Omega$, since $u_n-u_0 \in H^1_0(\Omega)$, and we can use the compact embedding of $H^1_0(\Omega)$ into $L^2(\Omega)$. If $\Omega$ had additional regularity (e.g. Lipschitz boundary), the weak lower semicontinuity result would hold for non-negative functions in $H^1(\Omega)$, not only functions in $K$.} described at the start of \cref{sec:approximating_problem}. 
We manipulate
\begin{align*}
    -\int_{\{u=0\}}F(\dist(x, \hat P_\rho(u)))\dx &=     -\int_{\{u=0\}\setminus E_\rho} F(\dist(x, \hat P_\rho(u)))\dx    \\& -\int_{\{u=0\} \cap E_\rho} F(\dist(x, \hat P_\rho(u))) \dx \\
    &\leqslant \liminf_{n \to \infty}-\int_{\{u_n=0\}\setminus E_\rho} F(\dist(x, \hat P_\rho(u_n)))\dx    + C_1\rho\tag{using \cref{prop:wlsc_like_rho} and \eqref{JISKND1}}\\
    &\leqslant \liminf_{n \to \infty}-\int_{\{u_n=0\}} F(\dist(x, \hat P_\rho(u_n)))\dx    + C_2\rho\tag{adding and subtracting the appropriate integral and using \eqref{JISKND1}}\\
        &\leqslant \liminf_{n \to \infty}-\int_{\{u_n=0\}} F(\dist(x, \hat P(u_n)))\dx    + C_2\rho
\end{align*}
for constants $C_1, C_2 > 0$, where we used \cref{lem:Prho_monotone} to obtain the final inequality. Taking the limit as $\rho \to 0$ on the both sides and using \cref{prop:distance_convergence} we get the result.
\end{proof}

Finally, we can prove the existence of minimisers.
\begin{proof}[Proof of \cref{thm:existence}]
Clearly $F(\dist(x,\partial Z(u)))  \leqslant F(\diam{\Omega})$ for every admissible $u$. Hence $J$ is bounded below on ${K}$. Thanks to \cref{thm:wlsc_Z_simple}, the rest of the argument is a standard application of the direct method. 
\end{proof}

\begin{remark}
    If in ${K}$ we omitted the boundary  condition that $u=u_0$ on $\partial\Omega$, we would need to enforce conditions on $f$ to show existence, see for example \cite{ShahgholianGustafsson}. The presence of $u_0$ helps because we can transform $\tilde u := u-u_0 \in H^1_0(\Omega)$ to rewrite the problem over $H^1_0(\Omega)$ and apply Poincaré's inequality to deal with the source term (to show that minimising sequences are bounded).
\end{remark}

\section{Stationarity and equations for the minimisers}\label{sec:stationarity}
In this section, we discuss the functional equations and inequalities satisfied
by the minimisers.
For this, we start with an obvious consequence of the definition of~$P(u)$ in~\eqref{PIUde}:

\begin{lemma}\label{OAKSlK}
If $v \geqslant w$ then $P(w) \subset P(v)$ and $\hat P(w) \subset \hat P(v)$.
\end{lemma}

Now we present a useful functional inequality enjoyed by minimisers:

\begin{prop}\label{P:3.2}
If $u$ is a minimiser of $J$, we have
\begin{equation}\label{3.30}
-\Delta u \leqslant f\chi_{\{u > 0\}} \quad \text{on $\Omega$}
\end{equation}
\end{prop}
\begin{proof}
Take $\varphi \in C_c^\infty(\Omega)$, $\varphi \geqslant 0$ and set $w=(u-t\varphi)^+$ for $t>0$. From the minimality of~$u$, we have that $J(u) \leqslant J(w)$, whence
\begin{align*}
0 &\leqslant J((u-t\varphi)^+) - J(u)\\
&= \int_\Omega \frac 12 |\nabla (u-t\varphi)^+|^2 - f(u-t\varphi)^+ - \frac 12|\nabla u|^2 + fu+
\int_{\{u=0\}} F(\dist(x, \hat P(u)))\dx\\
&\quad -\int_{\{(u-t\varphi)^+=0\}} F(\dist(x, \hat P((u-t\varphi)^+)))\dx\\
&\leqslant \int_\Omega \frac 12 |\nabla (u-t\varphi)|^2 - f(u-t\varphi)^+ - \frac 12|\nabla u|^2 + fu+
\int_{\{u=0\}} F(\dist(x, \hat P(u)))\dx\\
&\quad -\int_{\{(u-t\varphi)^+=0\}} F(\dist(x, \hat P((u-t\varphi)^+)))\dx.
\end{align*}
Now consider
\begin{equation}\label{EDGH91}
\frac 1t\int_\Omega  fu- f(u-t\varphi)^+ = \frac 1t\int_\Omega f\min(u,t\varphi)
\end{equation}
and note that pointwise a.e., if $x$ is such that $u(x)>0$, we have $u(x)-t\varphi(x) \geqslant u(x) - t\norm{\varphi}{\infty}$, and this will be non-negative for small enough $t$ (depending on~$x$), hence, on the set $\{ u > 0\}$, $\min(u,t\varphi)\slash t \to \varphi$ pointwise a.e. On the set $\{u=0\}$, $\min(u,t\varphi)\slash t = 0$. That is, $\min(u,t\varphi)\slash t \to \varphi\chi_{\{u > 0\}}$ a.e., and since $\min(u,t\varphi)\slash t \leqslant \varphi$, this convergence holds also in $L^p(\Omega)$ for all $p<\infty$. Hence the integral in~\eqref{EDGH91} converges to 
\[\int_\Omega f\varphi\chi_{\{u > 0\}}.\]
Finally, consider
\begin{equation}\label{EDGH912}
\int_{\{u=0\}} F(\dist(x, \hat P(u)))\dx -\int_{\{(u-t\varphi)^+=0\}} F(\dist(x, \hat P((u-t\varphi)^+)))\dx.
\end{equation}
Since $u \geqslant u-t\varphi$, we have $u^+ = u \geqslant (u-t\varphi)^+$, giving, by \cref{OAKSlK}, that
$\dist(x,\hat P(u)) \leqslant \dist(x, \hat P((u-t\varphi)^+))$. 

Using this, we estimate the quantity in~\eqref{EDGH912} by
\begin{equation}\begin{split}\label{PTHI}
&\int_\Omega \chi_{\{u=0\}} F(\dist(x, \hat P(u))) -\chi_{\{(u-t\varphi)^+=0\}} F(\dist(x, \hat P((u-t\varphi)^+)))\dx\\
&\leqslant \int_\Omega \left(\chi_{\{u=0\}}  -\chi_{\{(u-t\varphi)^+=0\}}\right) F(\dist(x, \hat P((u-t\varphi)^+)))\dx.
\end{split}\end{equation}
Now we observe that if~$ u(x)=0$ then~$(u(x)-t\varphi(x))^+ =
(-t\varphi(x))^+ =0$. Therefore,
$$ \chi_{\{u=0\}}  -\chi_{\{(u-t\varphi)^+=0\}}\leqslant 0.$$
Plugging this information into~\eqref{PTHI}, we find that
$$\int_\Omega \chi_{\{u=0\}} F(\dist(x, \hat P(u))) -\chi_{\{(u-t\varphi)^+=0\}} F(\dist(x, \hat P((u-t\varphi)^+)))\dx\leqslant 0.$$
These considerations show that
\[0 \leqslant \int_\Omega -\nabla u \nabla \varphi  + \int_\Omega f\varphi \chi_{\{u > 0\}}\]
and hence the desired result plainly follows.
\end{proof}

We observe that, if $f \leqslant 0$, \cref{P:3.2} gives that $u$ is subharmonic, and hence $r \mapsto \fint_{B_r(x)}u$ is non-decreasing in $r$, and so $\lim_{r \to 0}\fint_{B_r(x)}u$ exists for every $x$. This lets us replace the $\liminf$ in the definition of $P(u)$ with the limit. Secondly, we know that the above limit must coincide with $u(x)$ a.e. by the Lebesgue differentiation theorem, so we could choose a version of $u$ such that $\lim_{r \to 0}\fint_{B_r(x)}u = u(x)$ for every $x \in \Omega$
(and, in this setting, we actually have that $P(u) = \{ u > 0\}$).

We conclude this subsection by proving that the Poisson equation is solved by the minimiser on the interior of the positivity set.
\begin{proof}[Proof of \cref{prop:pde_solved_on_Pu}]
It suffices to show that
\begin{equation}\label{3.2.1}
-\Delta u \geqslant f  \quad \text{on $\mathrm{int}(P(u))$},
\end{equation}
since this, combined with \cref{lem:Pu_and_u_positive_same} and \eqref{3.30}, gives~\eqref{3.2.2}.

Thus, we focus on the proof of \eqref{3.2.1}.
For this, we define $Q: = \mathrm{int}(P(u))$ and take $\varphi \in C_c^\infty(Q)$ with $\varphi \geqslant 0$. We claim that, up to sets of measure zero,
\begin{equation}\label{HG67}
    \{u=0\} = \{ u + t\varphi = 0\}.
\end{equation}
Indeed, if~$u(x)+t\varphi(x)=0$, the fact that~$u$ and $\varphi$ are nonnegative gives that~$u(x)=\varphi(x)=0$, therefore~$ \{ u + t\varphi = 0\}$
is included in~$\{u=0\} $.

To show the reverse inclusion, suppose that $u(x)=0$.
Then, by~\cref{lem:Pu_cap_u_zero_measure_zero}, we know that~$x\not\in P(u)$,
up to a set of null measure that we can disregard here.

Hence, since $\varphi$ is supported in $Q \subset P(u)$, we obtain that~$\varphi(x)=0$. As a result, $u(x)+t\varphi(x)=0$, showing that~$\{u=0\} $
is included in~$ \{ u + t\varphi = 0\}$.
The proof of~\eqref{HG67} is thereby complete.

Also, the minimality of~$u$ gives that~$0 \leqslant J(u+t\varphi)-J(u)$. It follows from this and~\eqref{HG67} that
    \begin{align}
        0 &\leqslant \int_\Omega \frac 12 |\nabla (u+t\varphi)|^2 - \frac 12 |\nabla u|^2 - tf\varphi \nonumber \\
        &\quad + \int_{\{u=0\}} F(\dist(x, \hat P(u)))-F(\dist(x, \hat P(u+t\varphi)))\dx.\label{02ierkf1}
    \end{align}
Now we claim that \begin{equation}\label{02ierkf0}
P(u+t\varphi) = P(u).    
\end{equation}
To see this, observe first that the inclusion $P(u) \subset P(u+t\varphi)$ follows from the non-negativity of~$\varphi$. 

For the reverse inclusion, suppose that $x \in P(u+t\varphi)$. Then,
    \begin{equation}\label{02ierkf}
        \liminf_{r \to 0} \fint_{B_r(x)} u         
        = \liminf_{r \to 0} \fint_{B_r(x)} (u + t\varphi) - t\varphi(x)
        >- t\varphi(x).
    \end{equation}
     Suppose now, for the sake of contradiction, that $x \not\in P(u)$. Then,
     $$\liminf_{r \to 0} \fint_{B_r(x)} u=0.$$
Also, $\varphi(x)=0$ (since $\varphi$ is supported in~$Q\subset P(u)$)
and we thus deduce from~\eqref{02ierkf} that~$0>0$.
This contradiction completes the proof of~\eqref{02ierkf0}.

Therefore, in light of~\eqref{02ierkf0},
the two distance terms in~\eqref{02ierkf1} are equal, and we end up with $0 \leqslant \int_\Omega \nabla u \cdot \nabla \varphi - f\varphi,$ giving \eqref{3.2.1}, as desired. 
\end{proof}

\section{Free boundary condition}\label{sec:fbc}

In this section, we discuss the free boundary condition for the minimisers. To this end, we have to strengthen our assumptions by taking also \cref{ass:measure_Pu_zero}, i.e., we work with minimisers that satisfy $|\partial P(u)|=0$.

Now, for convenience, we split our energy $\Jfull = D + \Jd$  into the following two components:
\[ D(u) = \int_\Omega \left(\frac 12 |\nabla u|^2 - fu\right) \ \qquad \ \text{and} \ \qquad \Jd(u) = - \int_{\{u=0\}} F(\dist(x,\hat P(u)))\dx.\]
We define
\[
m_u(x):=\dist(x,\hat P(u)).
\]
Since $\hat P(u)$ need not be closed, it is useful to recall that
\[
m_u(x)=\dist(x,\overline{\hat P(u)}).
\]
Moreover, we have:
\begin{enumerate}[itemsep=0pt, label=(\roman*)]
    \item $m_u(x)\geqslant 0$ for all $x\in\Omega$;
    \item $m_u$ is $1$-Lipschitz;
    \item $m_u$ is differentiable a.e.\ in $\Omega$, and $\nabla m_u\in L^\infty(\Omega)$.
\end{enumerate}
We can write $\Jd$ as
\begin{equation}
    \label{J_en} \Jd(u) = - \int_{\{u=0\}} (F \circ m_u)(x)\dx.
\end{equation} 
Now, observe that since $F\colon \mathbb{R}_+ \to \mathbb{R}_+$ is assumed to be monotonic, it is differentiable almost everywhere with $F' \in L_{\rm{loc}}^1(\mathbb{R}).$ This useful fact will be employed later in \cref{DWLYUAS}. %

Recall that a function $\Phi\colon \Omega \to \Omega$ is said to be \emph{bi-Lipschitz} if there exists a constant $c \geqslant 1$ such that $$c^{-1}|x-y| \leqslant |\Phi(x) - \Phi(y)| \leqslant c|x-y|.$$

\begin{lemma}\label{lem_P_PHI_com}
    Assume $\Phi: \Omega \to \Omega$ is a bi-Lipschitz $C^1$-diffeomorphism. We define, on a non-negative function $u \in H^1(\Omega)$, $u_\Phi(x) = u(\Phi^{-1}(x))$. Then 
    \begin{equation}\label{P_PHI_com}
        P(u_\Phi) = \Phi(P(u)).
    \end{equation}
\end{lemma}
\begin{proof}
    By the bi-Lipschitz assumption, if $w = \Phi^{-1}(x)$,
     for all $r > 0$ 
    we have that
    \begin{equation}\label{INCL}B_{r/c}(w) \subset \Phi^{-1}(B_r(x)) \subset B_{cr}(w).\end{equation}
    Now, define 
    $$I_\Phi^r(x) := \dfrac{1}{|B_r|}\int_{\Phi^{-1}(B_r(x))} u(z){\rm det}(\nabla \Phi(z))\dz$$ 
    Since $\Phi$ is bi-Lipschitz we have $0 < b < {\rm det}(\nabla \Phi) < B$ for some constants $b,B \in \mathbb{R}$. 
    It follows from this and the fact that $u \geqslant 0$ that
    $$\dfrac{b}{|B_r|}\int_{\Phi^{-1}(B_r(x))}u(z)\dz \leqslant I_\Phi^r(x) \leqslant \dfrac{B}{|B_r|}\int_{\Phi^{-1}(B_r(x))}u(z)\dz.$$
Then, by~\eqref{INCL} and using $|B_{r/c}| = c^{-d}|B_r|$ and $|B_{cr}| = c^d|B_r|$, we find that
    $$bc^{-d}\fint_{B_{r/c}(x)}u(z)\dz \leqslant I_\Phi^r(x) \leqslant Bc^d\fint_{B_{cr}(x)}u(z)\dz.$$ Therefore,
    $$bc^{-d}\liminf_{r\to 0}\fint_{B_{r}(w)}u(z)\dz \leqslant \liminf_{r\to 0}I_\Phi^r(x) \leqslant Bc^d\liminf_{r\to 0}\fint_{B_{r}(w)}u(z)\dz.$$ 
    This inequality allows us to finish the proof due to the identity
    \begin{equation}\label{LIDEN}
        \fint_{B_r(x)} u_\Phi(y)\d y = \fint_{B_r(x)} u(\Phi^{-1}(y))\d y = \dfrac{1}{|B_r|}\int_{\Phi^{-1}(B_r(x))}u(z){\rm det}(\nabla\Phi(z))\dz = I_\Phi^r(x).
    \end{equation}
    Indeed, on the one hand, if $w = \Phi^{-1}(x) \in P(u)$ then $$\liminf_{r\to 0} \fint_{B_r(w)}u(z)\dz > 0,$$ hence $x \in P(u_\Phi)$, owing to~\eqref{LIDEN}, which shows that~$\Phi(P(u))$ is contained in~$P(u_\Phi) $.
    
      On the other hand, if $w = \Phi^{-1}(x) \notin P(u)$, then 
$$\liminf_{r\to 0} \fint_{B_r(w)}u(z)\dz = 0,$$ and hence, by \eqref{LIDEN}, we get $x \notin P(u_\Phi)$. This shows that~$P(u_\Phi)$ is a subset
of~$\Phi( P(u))$.
\end{proof}

Let now $\bphi \in C_0^\infty(\Omega, \mathbb{R}^d)$ and for $\varepsilon \in \mathbb{R}$ define 
\begin{equation}\label{eq:alphaepsphieps}    
\balpha_\varepsilon(x) = x + \varepsilon\bphi(x). 
\end{equation}
For $|\varepsilon| > 0$ small, $\balpha_\varepsilon$ is a diffeomorphism. Let us also define $u_\varepsilon$ through the implicit relation 
\begin{equation}
u_\varepsilon(\balpha_\varepsilon(x)) = u(x).\label{eq:uepsilondefn}
\end{equation}
Then, we have:

\begin{lemma}\label{lemma_alphae_com}
    For $u \in {K}$ we have $P(u_\varepsilon) = \balpha_\varepsilon(P(u))$ and hence $\overline{P(u_\varepsilon)} = \balpha_\varepsilon(\overline{P(u)})$. 
    
    In addition, $\hat P(u_\varepsilon) = \balpha_\varepsilon(\hat P(u))$ and hence $\overline{\hat P(u_\varepsilon)} = \balpha_\varepsilon(\overline{\hat P(u)})$ for sufficiently small $|\varepsilon|$. 
\end{lemma}
\begin{proof} Notice that $\balpha_\varepsilon$
satisfies the properties required on $\Phi$ in \cref{lem_P_PHI_com}. Hence, %
the first claim is a consequence of \cref{lem_P_PHI_com}. For the second claim, we use the fact that $\balpha_\varepsilon\colon \overline\Omega \to \overline\Omega$ is a diffeomorphism and hence maps boundaries to boundaries.
\end{proof}
  \cref{lemma_alphae_com} and~\cref{ass:measure_Pu_zero} imply that
\[
|\partial P(u_\varepsilon)|=0.
\]
This fact will be used in the proof of the following result.

\begin{prop}\label{lem:pointwise_derivative}
Let \cref{ass:measure_Pu_zero} hold for a minimiser $u$ of $J$. Let~$s\in\R$ be sufficiently small. Then, for a.e.\ $x\in\{u=0\}$,
the map 
$\varepsilon\mapsto m_{u_\varepsilon}(\balpha_\varepsilon(x))$ is differentiable at $s$ with
\[
\frac{d}{d\varepsilon}m_{u_\varepsilon}(\balpha_\varepsilon(x))\Big|_{\varepsilon=s}
=
\frac{\balpha_s(x)-\balpha_s(p_s(x))}{|\balpha_s(x)-\balpha_s(p_s(x))|}
\cdot(\bphi(x)-\bphi(p_s(x))),
\]
where $p_s(x)\in\overline{\hat P(u)}$ is the unique minimiser of 
\begin{equation}
\min_{z \in \overline{\hat P(u)}} |\balpha_s(x)-\balpha_s(z)|.    \label{eq:p_sx}
\end{equation}
\end{prop}
\begin{proof}
 We have
\[
    m_{u_{\varepsilon}}(\balpha_{\varepsilon}(x)) = \inf_{z \in \overline{\hat P(u_\varepsilon)}}|\balpha_{\varepsilon}(x) - z| = \inf_{z\in \overline{\hat P(u)}}|\balpha_{\varepsilon}(x) - \balpha_{\varepsilon}(z)|  =\inf_{z\in \overline{\hat P(u)}} g_x(\varepsilon, z),
\]
where \begin{equation}\label{LP}
    g_x(\varepsilon, z) = |\balpha_{\varepsilon}(x) - \balpha_{\varepsilon}(z)|.
\end{equation}
Let us consider from now on points $x$ with
\[x \in \{u=0\} \setminus (\partial  P(u) \cup M)\]
where $M$ is the null set on which $P(u)$ and $\{u > 0\}$ differ (indeed, we know these two sets are the same up to null sets, by \cref{lem:Pu_and_u_positive_same}).

We verify the hypotheses of Danskin's theorem
(see e.g.~\cite[Theorem 2.4]{MR4647544}):
\begin{enumerate}[itemsep=0pt, label=(\roman*)]
    \item \emph{Differentiability of the objective.} First, clearly $g_x$ is continuous. 
In addition, from \eqref{LP}, we know that $g_x$ is differentiable when $\balpha_{\varepsilon}(x) \neq  \balpha_{\varepsilon}(z)$, that is, if~$|\varepsilon|$ is small enough, when $x \neq z$, with
    \[
        \partial_{\varepsilon} g_x(\varepsilon, z)
        =
        \frac{\balpha_{\varepsilon}(x)-\balpha_{\varepsilon}(z)}{|\balpha_{\varepsilon}(x)-\balpha_{\varepsilon}(z)|}
        \cdot(\bphi(x)-\bphi(z)).
    \] 

    Since $x\in\{u=0\}\setminus M$, $x\notin P(u)$; as also $x\notin\partial P(u)$ by hypothesis, $x\notin\overline{P(u)}$. Since $x\in\{u=0\}\subset\Omega$, $x\notin\partial\Omega$. Hence $x\notin\overline{P(u)}\cup\partial\Omega=\overline{\hat P(u)}$, so $x$ cannot equal any $z\in\overline{\hat P(u)}$.
    
    Hence for all $x$ as considered, $g_x$ is differentiable with respect to $\varepsilon$, with a continuous derivative.

    \item \emph{Uniqueness of the minimiser.} 
Now let us show that the set
\[ \left\{z \in \argmin\limits_{z \in \overline{\hat P(u)}} g_x(\varepsilon, z)\right\}\]
is a singleton. 

To this end, we point out that 
    since $\dist(\cdot,\overline{\hat P(u_\varepsilon)})$ is Lipschitz, it is differentiable outside a set $N_\varepsilon$ of null measure, and hence (see e.g. \cite[Proposition 2.1]{borrelli2026submanifoldsclassc1alphasets}) there exists a unique nearest point in $\overline{\hat P(u_\varepsilon)}$ to $y$, for all $y \in \overline{\hat P(u_\varepsilon)}^c \cap N_\varepsilon^c$.  %
That is, the projection operator is well defined for such points,
and for all $ y \in \overline{\hat P(u_\varepsilon)}^c \cap N_\varepsilon^c$
there exists a unique $ y_\varepsilon^* $ for which \[ \dist(y,\overline{\hat P(u_\varepsilon)}) = |y-y_\varepsilon^*| .\] %
Now, if $x \in \{u=0\} \setminus \partial P(u)$ then $\balpha_\varepsilon(x) \in \{u_\varepsilon=0\}\setminus \partial P(u_\varepsilon)$, thus $\balpha_\varepsilon(x) \not\in \overline{P(u_\varepsilon)}.$ 

Since $\balpha_\varepsilon(x) \in \{u_\varepsilon = 0\} \subset \Omega$, $\balpha_\varepsilon(x) \not\in \partial \Omega$. Thus $\balpha_\varepsilon(x) \in \overline{P(u_\varepsilon)}^c \cap \partial\Omega^c = (\overline{P(u_\varepsilon)} \cup \partial\Omega)^c = \overline{\hat P(u_\varepsilon)}^c$. 

Also, since $\balpha_\varepsilon$ and its inverse function send null sets into null sets, it follows that $\balpha_\varepsilon(x) \not\in N_\varepsilon$ for $x \in \{u=0\}\setminus (\partial P(u) \cup \balpha_\varepsilon^{-1}N_\varepsilon)$. 

In this way, applying the above displayed equation, we have that $\dist(\balpha_\varepsilon(x),\overline{\hat P(u_\varepsilon)})$ is achieved by a unique element $q(x) \in  \overline{\hat P(u_\varepsilon)}$, with
\[\dist(\balpha_\varepsilon(x),\overline{\hat P(u_\varepsilon)}) = |\balpha_\varepsilon(x) - q(x)|.\]

Thus,
\begin{align*}&
\inf_{y \in \overline{\hat P(u)}} |\balpha_\varepsilon(x)-\balpha_\varepsilon(y)| = \inf_{y \in \overline{\hat P(u_\varepsilon)}} |\balpha_\varepsilon(x)-y| \\&\quad= \dist(\balpha_\varepsilon(x), \overline{\hat P(u_\varepsilon)}) = |\balpha_\varepsilon(x)- q(x)| = |\balpha_\varepsilon(x)- \balpha_\varepsilon p_\varepsilon(x)|,
\end{align*}
where we set 
\[p_\varepsilon(x) := \balpha_\varepsilon^{-1}q(x) \in \overline{\hat P(u)}.\]
That is, $p_\varepsilon(x)$ is a minimiser of the functional on the left-hand side. It is unique: if it were not unique, say also $p^* \neq p_\varepsilon(x)$ is a minimiser, we would have
\begin{align*}
|\balpha_\varepsilon(x)-\balpha_{\varepsilon}(p^*)| &= \inf_{y \in \overline{\hat P(u)}} |\balpha_\varepsilon(x)-\balpha_\varepsilon(y)|\\
& = \inf_{y \in \overline{\hat P(u_\varepsilon)}} |\balpha_\varepsilon(x)-y| = |\balpha_\varepsilon(x)- q(x)| 
\end{align*}
which shows that $\balpha_\varepsilon(p^*)$ and $q(x)$ are minimisers of the functional over $\overline{\hat P(u_\varepsilon)}$. By uniqueness of the projection at the point $\balpha_\varepsilon(x)$, we must have $q(x)=\balpha_\varepsilon(p^*)$ but $q(x)=\balpha_\varepsilon p_\varepsilon(x)$ by definition, which implies that $p_\varepsilon(x) = p^*$ as $\balpha_\varepsilon$ is a diffeomorphism, a contradiction. Thus we have shown that
\[ \inf_{y \in \overline{\hat P(u)}} g_x(\varepsilon, y)\]
has the unique minimiser $p_\varepsilon(x)$, for all $x \in \{u=0\}\setminus (\partial P(u)\cup \balpha_\varepsilon^{-1}N_\varepsilon)$.

\end{enumerate}

Danskin's theorem \cite[Theorem 2.4]{MR4647544}
thus gives that
\[
    \frac{d}{d\varepsilon}m_{u_\varepsilon}(\balpha_\varepsilon(x))
    = \partial_{\varepsilon} g_x(\varepsilon, p_\varepsilon(x))
    = \frac{\balpha_\varepsilon(x)-\balpha_\varepsilon(p_\varepsilon(x))}{|\balpha_\varepsilon(x)
    -\balpha_\varepsilon(p_\varepsilon(x))|}
    \cdot(\bphi(x)-\bphi(p_\varepsilon(x)))
\]
for all $x \in \{u=0\}\setminus (\partial P(u)\cup M \cup \balpha_\varepsilon^{-1}N_\varepsilon)$.

\end{proof}
We point out that, at $\varepsilon=0$, the minimiser $p_0(x)=p(x)$ of \eqref{eq:p_sx} is the nearest point in 
$\overline{\hat P(u)}$ to $x$, and $\frac{x-p(x)}{|x-p(x)|}=\nabla m_u(x)$ by the 
standard formula for the gradient of the distance function. Hence we have
\[
    \frac{d}{d\varepsilon}m_{u_\varepsilon}(\balpha_\varepsilon(x))\Big|_{\varepsilon=0}
    = \nabla m_u(x)\cdot(\bphi(x)-\bphi(p(x))). 
\]
where
\begin{equation}
p(x) := \argmin_{z \in \overline{\hat P(u)}} |x-z|.  \label{eq:p0x}    
\end{equation}
Now consider the map $\varepsilon \mapsto \BJ(\varepsilon) := \Jd(u_\varepsilon)$ as a function from $[-1,1]$ to $\mathbb{R}$:
$$\BJ(\varepsilon) = -\int_{\{u_\varepsilon = 0\}} (F \circ m_{u_\varepsilon})(x)\dx = -\int_{\{u=0\}} (F \circ m_{u_\varepsilon} \circ \balpha_\varepsilon) {{\rm det}} \nabla \balpha_\varepsilon \dx.$$   
\begin{prop}\label{DWLYUAS}
Let \cref{ass:measure_Pu_zero} hold for a minimiser $u$ of $J$.
The map $\BJ$ is differentiable at $0$ and
\begin{equation}\label{eq:BJ-prime}
\BJ'(0)
=
-\int_{\{u=0\}}
(F'\circ m_u)\,\nabla m_u\cdot(\bphi(x)-\bphi(p(x)))
+
(F\circ m_u)\,\grad \cdot \bphi(x)
 \dx
\end{equation}
where $p(x)$ is as in \eqref{eq:p0x}.
\end{prop}

\begin{proof}
Write
\[
\BJ(\varepsilon)
=
-\int_{\{u=0\}} G_\varepsilon(x)\dx,
\quad
G_\varepsilon(x)
=
F(m_{u_\varepsilon}(\balpha_\varepsilon(x)))\det D\balpha_\varepsilon(x).
\]
We apply the dominated convergence theorem to differentiate under the integral sign.
We bound the difference quotient directly:
\begin{align*}
    \left|\frac{G_\varepsilon(x)-G_0(x)}{\varepsilon}\right|
    &\leqslant \left|\frac{F(m_{u_\varepsilon}(\balpha_\varepsilon(x))) - 
          F(m_u(x))}{\varepsilon}\right| |\det D\balpha_\varepsilon|
     + |F(m_u(x))| \left|\frac{\det D\balpha_\varepsilon - 1}{\varepsilon}\right|.
\end{align*} For the second term, since $D\balpha_\varepsilon = I + \varepsilon D\bphi$, the map 
$\varepsilon \mapsto \det(I+\varepsilon D\bphi(x))$ is a polynomial in $\varepsilon$ 
of degree $d$. Hence $\det(I+\varepsilon D\bphi(x)) - 1$ has no constant term, so 
$\frac{\det(I+\varepsilon D\bphi(x))-1}{\varepsilon}$ is itself a polynomial in 
$\varepsilon$ with coefficients depending only on $D\bphi(x)$, and is therefore 
bounded uniformly in $x$ and small $\varepsilon$ by a constant depending only on 
$\|D\bphi\|_{L^\infty}$. For the first term, we observe that, for all functions $f$ and $g$ and $\eta>0$,
\begin{eqnarray*}
&& \inf_{z\in\overline{\hat P(u)}} |f(z)|-\inf_{z\in\overline{\hat P(u)}}| g(z)|=
\inf_{z\in\overline{\hat P(u)}} |f(z)|-\inf_{w\in\overline{\hat P(u)}}| g(w)|
\\&&\qquad\leqslant |f(w_\eta)|-|g(w_\eta)|+\eta
\leqslant |f(w_\eta)-g(w_\eta)|+\eta\leqslant\sup_{z\in\overline{\hat P(u)}} |f(z)-g(z)|+\eta,
\end{eqnarray*}
for a suitable $w_\eta\in\overline{\hat P(u)}$.

Thus, taking~$\eta$ as small as we wish and possibly exchanging the roles of~$f$ and $g$,
$$\left| \inf_{z\in\overline{\hat P(u)}} |f(z)|-\inf_{z\in\overline{\hat P(u)}}| g(z)|\right|\leqslant \sup_{z\in\overline{\hat P(u)}} |f(z)-g(z)|.$$
Using this observation and the fact that $\hat P(u_\varepsilon) = \balpha_\varepsilon(\hat P(u))$, we find that
\begin{align*}
    |m_{u_\varepsilon}(\balpha_\varepsilon(x)) - m_u(x)|
    &= \left|\inf_{z\in\overline{\hat P(u)}}|\balpha_\varepsilon(x)-\balpha_\varepsilon(z)| 
       - \inf_{z\in\overline{\hat P(u)}}|x-z|\right| \\
    &\leqslant \sup_{z\in\overline{\hat P(u)}} \left||\balpha_\varepsilon(x)-\balpha_\varepsilon(z)| 
       - |x-z|\right| \\
    &\leqslant \sup_{z\in\overline{\hat P(u)}} |\varepsilon(\bphi(x)-\bphi(z))|\\&
    \leqslant 2\|\bphi\|_{L^\infty}|\varepsilon|.
\end{align*}
Since $F$ is Lipschitz on $[0,\diam{\Omega}]$ with constant 
$\|F'\|_{L^\infty([0,\diam{\Omega}])}$, this gives
\[
    \left|\frac{F(m_{u_\varepsilon}(\balpha_\varepsilon(x))) - 
    F(m_u(x))}{\varepsilon}\right| 
    \leqslant \|F'\|_{L^\infty} \cdot 2\|\bphi\|_{L^\infty},
\]
uniformly in $x$ and $\varepsilon$. 

Thus, since $|\det D\balpha_\varepsilon|$ is also 
uniformly bounded, 
\[
    \left|\frac{G_\varepsilon(x)-G_0(x)}{\varepsilon}\right| \leqslant C
\]
for a constant $C$ independent of $x$ and small $\varepsilon$. Since $|\Omega|<\infty$, 
this is an $L^1(\{u=0\})$ dominating function, and the dominated convergence theorem gives
\[
\BJ'(0)
=
-\int_{\{u=0\}} \partial_\varepsilon G_\varepsilon(x)\big|_{\varepsilon=0}\dx.
\]
It remains to compute $\partial_\varepsilon G_\varepsilon(x)|_{\varepsilon=0}$. 
By \cref{lem:pointwise_derivative} at $\varepsilon=0$,
\[
\frac{d}{d\varepsilon} m_{u_\varepsilon}(\balpha_\varepsilon(x))\Big|_{\varepsilon=0}
=
\nabla m_u(x)\cdot(\bphi(x)-\bphi(p(x)))
\]
for a.e.\ $x\in\{u=0\}$. Using also $\det D\balpha_0 = 1$ and 
$\partial_\varepsilon \det D\balpha_\varepsilon|_{\varepsilon=0} = \operatorname{div}\bphi$, 
we obtain
\[
\partial_\varepsilon G_\varepsilon(x)\big|_{\varepsilon=0}
=
(F'\circ m_u)\,\nabla m_u\cdot(\bphi(x)-\bphi(p(x)))
+
(F\circ m_u)\,\operatorname{div} \bphi,
\]
and substituting gives \eqref{eq:BJ-prime}. In this computation, we used the fact that $F'(m_u(x))$ exists for a.e. $x\in\{u=0\}$. Indeed, as mentioned before,  $F'$ exists for every point outside a null set $N$. Now, using that $|\nabla m_u(x)| = 1$ on $\{u=0\}$ (thanks to \cref{ass:measure_Pu_zero} and because $|\nabla m_u| = 1$ a.e. on $\mathbb{R}^d \setminus \overline{\hat P(u)}$ by e.g. \cite[Remark 2.1]{borrelli2026submanifoldsclassc1alphasets}) and applying the coarea formula,
\begin{align*}
|\{x\in\{u=0\}: m_u(x)\in N\}|  &= 
\int_{\{u=0\}} \chi_N(m_u(x))\,|\nabla m_u(x)|\dx\\
&= \int_{-\infty}^\infty \left(\int_{\{u=0\}\cap m_u^{-1}(t)} \chi_N(m_u(x))\d\mathcal H^{d-1}(x)\right) \d t\\
&=  \int_{-\infty}^\infty \chi_N(t)\left(\int_{\{u=0\}\cap m_u^{-1}(t)} \d\mathcal H^{d-1}(x)\right) \d t\\
& = 0
\end{align*}
since $N$ is a null set. Hence for almost all $x$, $F'(m_u(x))$ is well defined.
\end{proof}

 In the next result, note that the integral over $\{u>0\}$ on the left-hand side can be written equivalently over $\Omega$, due to Stampacchia's lemma.
\begin{prop}[First domain variation formula]\label{lem:FDVFormula}
Let \cref{ass:measure_Pu_zero} hold  for a minimiser $u$ of $J$.
Then for every $\bphi \in C_0^\infty(\Omega, \mathbb{R}^d)$ the following identity holds: \begin{align}
        &\int_{\{u>0\}} \left( \frac 12|\nabla u|^2{\divop}\bphi - \nabla u \cdot \nabla \bphi \nabla u + f \nabla u \cdot \bphi  \right)\dx \nonumber \\ 
        &=  \int_{\{u=0\}} \left((F' \circ m_u) \nabla m_u \cdot \left[\bphi - \bphi(p(x))\right] + (F \circ m_u) {\divop} \bphi \right)\dx\label{first_var_cn2}  
    \end{align} 
    where %
    $p(x)$ is as in \eqref{eq:p0x}.
\end{prop}
\begin{proof}
As $\varepsilon \mapsto \Jfull(u_\varepsilon)$ is minimal at $\varepsilon =0$, we have $\frac{d}{d\varepsilon}\Jfull(u_\varepsilon)|_{\varepsilon=0}=0$. We now work this out explicitly. Recalling again $\balpha_\varepsilon$ and $u_\varepsilon$ from \eqref{eq:alphaepsphieps} and \eqref{eq:uepsilondefn}, we have
    \begin{align}
         &D(u_\varepsilon) - D(u) \nonumber \\ 
        &= \int_{\{u>0\}}\left[\frac 12|(\nabla \balpha_\varepsilon^\top)^{-1}\nabla u|^2  - (f \circ \balpha_\varepsilon) u \right] {{\rm det}} \nabla \balpha_\varepsilon \dx  -  \int_{\{u>0\}}\left[\frac 12|\nabla u|^2 - fu\right] \dx. \label{e0}
    \end{align}
    From $u_\varepsilon \circ \balpha_\varepsilon = u$ we have $\nabla \balpha_\varepsilon^\top\nabla u_\varepsilon = \nabla u$ whereby $\nabla u_\varepsilon = (\nabla \balpha_\varepsilon^\top)^{-1}\nabla u $. 
    
    We also have that $\nabla \balpha_\varepsilon = I + \varepsilon \nabla \bphi$ and $(\nabla \balpha_\varepsilon)^{-1} = I - \varepsilon \nabla \bphi + o(\varepsilon)$, hence $(\nabla \balpha_\varepsilon^\top)^{-1} = I - \varepsilon \nabla \bphi^\top + o(\varepsilon)$.

    Thus,
    $$(\nabla \balpha_\varepsilon^\top)^{-1}\nabla u = (I-\varepsilon \nabla \bphi^\top + o(\varepsilon))\nabla u = \nabla u - \varepsilon \nabla \bphi^\top \nabla u + o(\varepsilon)$$and$$|(\nabla \balpha_\varepsilon^\top)^{-1} \nabla u|^2 = |\nabla u|^2 - 2 \varepsilon \nabla u \cdot \nabla \bphi \nabla u + o(\varepsilon),$$
while $${{\rm det}} \nabla \balpha_\varepsilon = 1 + \varepsilon {\divop} \bphi + o(\varepsilon).
   $$
   Hence \eqref{e0} becomes  
   \begin{align*} 
        &D(u_\varepsilon) - D(u)  = \varepsilon \int_{\{u>0\}} \left[ \frac 12|\nabla u|^2 {\divop}\bphi - \nabla u \cdot \nabla \bphi \nabla u  \right]\dx  +  \varepsilon \int_{\{u>0\}}  f \nabla u \cdot \bphi  \dx + o(\varepsilon),
    \end{align*} whereby 
    \begin{align*}
        \dfrac{\partial}{\partial\varepsilon}\left[D(u_\varepsilon)\right]\biggr\rvert_{\varepsilon = 0} =  \int_{\{u>0\}} \left[\frac 12|\nabla u|^2 {\divop}\bphi - \nabla u \cdot \nabla \bphi \nabla u \right]\dx + \int_{\{u>0\}} f \nabla u \cdot \bphi \dx.
    \end{align*}
  The desired result now follows from \cref{DWLYUAS}.
    \end{proof}

\begin{proof}[Proof of \cref{thm:fbc}]
    \cref{prop:pde_solved_on_Pu} implies $-\Delta u = f$ a.e. on $\{u > 0\}$, which gives
    \begin{align*}
        \frac 12|\nabla u|^2{\divop}\bphi  - \nabla u \cdot \nabla \bphi \nabla u +f \nabla u \cdot \bphi &= \frac 12|\nabla u|^2{\divop}\bphi - \nabla u \cdot \nabla \bphi \nabla u - (\Delta u) \nabla u \cdot \bphi \\ &= \divop \left(\frac 12|\nabla u|^2\bphi-(\nabla u \cdot \bphi)\nabla u\right).
    \end{align*}
Hence, using \cref{lem:FDVFormula},
\begin{eqnarray*}&&
       \int_{\{u > 0\}} \divop \left(\frac 12|\nabla u|^2\bphi-(\nabla u \cdot \bphi)\nabla u\right) \\&&\qquad =\int_{\{u=0\}} \left[(F' \circ m_u) \nabla m_u \cdot \left[\bphi - \bphi(x_p)\right] + (F \circ m_u) {\divop} \bphi \right]\dx.
    \end{eqnarray*}
   Now, the left-hand side is
\begin{align*}
       \int_{\{u > 0\}} \divop \left(\frac 12|\nabla u|^2\bphi-(\nabla u \cdot \bphi)\nabla u\right) &= 
       \lim_{\varepsilon \downarrow 0}\int_{\{u > \varepsilon\}} \divop \left(\frac 12|\nabla u|^2\bphi-(\nabla u \cdot \bphi)\nabla u\right)\\
       &= \lim_{\varepsilon \downarrow 0}\int_{\partial\{u > \varepsilon\}} \left(\frac 12|\nabla u|^2\bphi-(\nabla u \cdot \bphi)\nabla u\right)\cdot \nu_\varepsilon\\
       &= \lim_{\varepsilon \downarrow 0}\int_{\partial\{u > \varepsilon\}} - \frac 12|\nabla u|^2\bphi\cdot \nu_\varepsilon
       \end{align*}
       where we used the expression for the outward normal $\nu_\varepsilon = -\frac{\grad u}{|\grad u|}$ since $\partial\{u > \varepsilon\} = \{u = \varepsilon\}$. Combining the previous two displayed equations gives the first equality in \eqref{fb_0}. 

    For the  second, recall the product rule: ${\divop}(\psi\Psi) = \psi {\divop}(\Psi) + \nabla \psi \cdot \Psi.$ We apply this for $\psi = F \circ m_u$ and $\Psi(x) = \bphi(x)$, finding that \begin{align*}
        (F\circ m_u){\divop}( \bphi) + (F' \circ m_u)\nabla m_u \cdot \bphi = {\divop} \left( (F\circ m_u) \bphi\right).
    \end{align*}
Then we have, using the decomposition $\Omega = \{ u = 0\} \cup \{ u > 0\}$,
\begin{align*}
\int_{\{u=0\}} \divop \left( (F\circ m_u) \bphi\right) &= \int_\Omega \divop \left( (F\circ m_u) \bphi\right)  - \lim_{\varepsilon \downarrow 0}\int_{\{u > \varepsilon\}}  \divop\left( (F\circ m_u) \bphi\right) \\
&= - \lim_{\varepsilon \downarrow 0}\int_{\partial\{u > \varepsilon\}}  (F\circ m_u) \bphi\cdot \nu_\varepsilon
\end{align*}
with the final equality as $\bphi$ is compactly supported.
    
The second formula in \eqref{fb_0} then follows.
\end{proof} 

\begin{remark}
We assumed $C^1$-smoothness of $\partial\{u > \varepsilon\}$ in order to use the divergence theorem and to approximate the positivity set. This assumption could be weakened; however, it is standard in the literature.
\end{remark}

 We conclude now with proving the pointwise free boundary conditions in a one-dimensional setting.
\begin{proof}[Proof of \cref{cor:1D_FB}]
Write
\[
Z(u)
=
(0,1)\setminus\bigcup_{i=1}^k(a_i,b_i).
\]
Let $(c,d)$ represent a non-empty zero-gap of $Z(u)$ (that is, a maximal closed or half-closed interval contained in $Z(u)$), so that
$c,d\in\{0,1,a_i,b_i\}_{i=1}^k$ and $c < d$. By \cref{lem:Z_reduction}, for $x \in (c,d)$,
\[
m_u(x)=\dist(x,\partial Z(u))
=
\min(x-c,d-x),
\]
For
$x\in(c,(c+d)/2)$,
\[
m_u(x)=x-c,\qquad
m_u'(x)=1,\qquad
x_p=c.
\]
Hence, taking $\bphi\in C_0^1(\Omega,\mathbb R)$,
\[
\frac{d}{dx}
\left(
F(x-c)(\bphi(x)-\bphi(c))
\right)
=
(F'\circ m_u)m_u'
(\bphi-\bphi(x_p))
+
(F\circ m_u)\bphi',
\]
and therefore
\[
\int_c^{(c+d)/2}
(F'\circ m_u)m_u'
(\bphi-\bphi(x_p))
+
(F\circ m_u)\bphi'
\dx
=
F\left(\frac{d-c}{2}\right)
\left(
\bphi\left(\frac{c+d}{2}\right)-\bphi(c)
\right).
\]
Similarly, for
$x\in((c+d)/2,d)$,
\[
m_u(x)=d-x,\qquad
m_u'(x)=-1,\qquad
x_p=d.
\]
Consequently,
\[
\frac{d}{dx}
\left(
F(d-x)(\bphi(x)-\bphi(d))
\right)
=
(F'\circ m_u)m_u'
(\bphi-\bphi(x_p))
+
(F\circ m_u)\bphi',
\]
and so
\[
\int_{(c+d)/2}^{d}
(F'\circ m_u)m_u'
(\bphi-\bphi(x_p))
+
(F\circ m_u)\bphi'
\dx
=
-
F\left(\frac{d-c}{2}\right)
\left(
\bphi\left(\frac{c+d}{2}\right)-\bphi(d)
\right).
\]
Adding the above two integral identities gives
\begin{equation}\label{eq:1D_gap}
\int_c^d
(F'\circ m_u)m_u'
(\bphi-\bphi(x_p))
+
(F\circ m_u)\bphi'
\dx
=
F\left(\frac{d-c}{2}\right)
(\bphi(d)-\bphi(c)).
\end{equation}
Summing \eqref{eq:1D_gap} over all  zero-gaps $(c,d)$ of $Z(u)$ yields
\[
\int_{\{u=0\}}
(F'\circ m_u)m_u'
(\bphi-\bphi(x_p))
+
(F\circ m_u)\bphi'
\dx
=
\sum_{(c,d)}
F\left(\frac{d-c}{2}\right)
(\bphi(d)-\bphi(c)),
\]
Substituting this identity into the derived free boundary condition \eqref{fb_0} of \cref{thm:fbc}, which we recall is
\[
\lim_{\varepsilon\downarrow0}
\int_{\partial\{u>\varepsilon\}}
\frac12(u')^2 \bphi\nu
=
-
\int_{\{u=0\}}
\left(
(F'\circ m_u)m_u'
(\bphi-\bphi(x_p))
+
(F\circ m_u)\bphi'
\right)
\dx,
\]
we obtain
\[
\sum_{i=1}^k
\left(
-\frac12u'(a_i)^2\bphi(a_i)
+
\frac12u'(b_i)^2\bphi(b_i)
\right)
=
-
\sum_{(c,d)}
F\left(\frac{d-c}{2}\right)
(\bphi(d)-\bphi(c)).
\]
Since $\bphi\in C_0^1(\Omega,\mathbb R)$, we have
$\bphi(0)=\bphi(1)=0$, so the endpoints $0$ and $1$ do not contribute. As the free boundary points are distinct, we may choose test functions supported in arbitrarily small neighbourhoods of each one. 
Comparing the coefficients of the values of $\bphi$ at each free boundary point yields the desired statement.

\end{proof}

\section{Conclusion}
We studied an interesting and novel functional involving the distance to the free boundary term, and answered the most fundamental questions of interest. Further work includes deriving pointwise free boundary conditions in dimensions higher than $1$, and regularity of solutions, in particular, establishing continuity of solutions taking into account the difficulties posed by the non-local term.

\section*{Acknowledgements}
AA thanks Constantin Christof for useful comments on an earlier version which pointed out in particular the well-definedness issue.

\bibliographystyle{abbrv}   
\bibliography{bib}

\end{document}